\documentclass{amsart}
\allowdisplaybreaks[1]
\usepackage{mathtools} \mathtoolsset{showonlyrefs}
\usepackage{verbatim}
\usepackage{enumitem}
\usepackage{placeins}

\usepackage{tikz}
\usetikzlibrary{angles, backgrounds, calc, intersections}
\usetikzlibrary{shapes.geometric}

\usepackage{hyperref}

\newtheorem{theorem}{Theorem}
\newtheorem{lemma}[theorem]{Lemma}
\newtheorem{corollary}[theorem]{Corollary}
\newtheorem{proposition}[theorem]{Proposition}
\newtheorem*{remark}{Remark}
\newtheorem{definition}[theorem]{Definition}

\def \R{\mathbb R}

\def \II{\mathrm{I\!I}}

\newcommand{\sumTilde}[1][-p]{
	\mathbin{
		\widetilde{+}_{#1}
	}
}
\newcommand {\vol}[1]{\operatorname{vol}\left(#1\right)}

\title{A Smooth Covariogram and the Transition Between Projection Bodies}

\author{Juli\'an Haddad}
\author{C. Hugo Jim\'enez}
\author{Matheus Resende}
\date{\today}

\begin{document}

	\begin{abstract}
We introduce a smooth radial approximation of the covariogram of a convex body.
The construction replaces the intersection \(K\cap(K+x)\) by the radial
\((-p)\)-mean of \(K\) and \(K+x\), producing a family \(g_{K,p}\) which converges
pointwise to the classical covariogram as \(p\to\infty\). For each fixed \(p>0\),
the function \(g_{K,p}\) is smooth near the origin, and its Hessian is computed
explicitly in terms of the operator defining the polar \(L_2\)-projection body.

We study the simultaneous limit in which \(p\to\infty\) and the level deficit is
\(\delta=p^\alpha\), \(\alpha<0\). The normalized level sets exhibit three
regimes: the polar projection body appears for \(\alpha>-1\), the polar
\(L_2\)-projection body appears for \(\alpha<-1\), and a critical transition body
appears at \(\alpha=-1\). This transition body is described by an explicit \(\log\cosh\)-type boundary
integral, placing the critical regime naturally within the framework of Orlicz
projection bodies.
\end{abstract}
	
	\maketitle
%

	\section{Introduction}

Let \(K\subset \mathbb R^n\) be a convex body, i.e. a compact convex set with
non-empty interior. The covariogram of \(K\) is the function
\[
g_K(x)=\operatorname{vol}(K\cap (K+x)),\qquad x\in\mathbb R^n.
\]
Equivalently,
\[
g_K=\mathbf{1}_K\ast \mathbf{1}_{-K},
\]
since
\[
(\mathbf{1}_K\ast \mathbf{1}_{-K})(x)
=
\int_{\mathbb R^n}\mathbf{1}_K(y)\mathbf{1}_{-K}(x-y)\,dy
=
\operatorname{vol}(K\cap(K+x)).
\]
Thus the covariogram is the convolution square of the indicator function of
\(K\), up to reflection, and measures the volume of overlap between \(K\) and its
translates.

The covariogram was introduced by Matheron in stochastic geometry and integral
geometry \cite{Matheron1975}. It has since become a central object in geometric
tomography, convex geometry, and related inverse problems. One of the most
classical questions in this direction is Matheron's covariogram problem, which
asks whether \(g_K\) determines \(K\), up to translations and reflections
\cite{Matheron1986}. This problem has generated a substantial literature; see,
for instance, \cite{Bianchi2005,AverkovBianchi2009,Bianchi2023}. The covariogram
is also closely related to phase retrieval, since its Fourier transform satisfies
\[
\widehat{g_K}=|\widehat{\mathbf{1}_K}|^2;
\]
see \cite{BianchiGardnerKiderlen2011}.

In the present paper we are not concerned with uniqueness questions. Instead, we
study the local geometry of the level sets of \(g_K\) near its maximum. Since
\[
g_K(0)=\operatorname{vol}(K),
\]
the sets
\[
\{x\in\mathbb R^n:g_K(x)\geq\operatorname{vol}(K)-\delta\}
\]
describe the behavior of the covariogram close to the origin. This behavior is
known to be connected to the polar projection body \(\Pi^\circ K\). More
precisely, as \(\delta\to0^+\), these level sets converge, after the
normalization by \(\delta\), to \(\Pi^\circ K\); see
\cite{kiener1986extremalitat,schmuckenschlager1992distribution}. This is one of the links between the
local behavior of the covariogram and affine convex geometry.

We recall the normalization used throughout the paper. The polar projection body
\(\Pi^\circ K\) is the unit ball of the norm
\[
\|u\|_{\Pi^\circ K}
=
h_{\Pi K}(u)
=
\frac12\int_{\partial K}|u\cdot n_K(x)|\,dS(x),
\qquad u\in\mathbb R^n.
\]
Similarly, the polar \(L_2\)-projection body \(\Pi_2^\circ K\) is the unit ball
of the quadratic form
\[
\|u\|_{\Pi_2^\circ K}^2
=
\int_{\partial K}
\frac{|u\cdot n_K(x)|^2}{x\cdot n_K(x)}
\,dS(x),
\qquad u\in\mathbb R^n.
\]
Equivalently,
\[
\Pi_2^\circ K
=
\{u\in\mathbb R^n:u^t\mathcal{M}_Ku\le1\},
\]
where
\[
\mathcal{M}_K
=
\int_{\partial K}
\frac{n_K(x)\otimes n_K(x)}
{x\cdot n_K(x)}
\,dS(x).
\]
These bodies satisfy the natural linear covariance relations
\[
\Pi^\circ(TK)=|\det T|^{-1}T(\Pi^\circ K),
\qquad
\Pi_2^\circ(TK)=|\det T|^{-1/2}T(\Pi_2^\circ K),
\]
for every \(T\in GL(n)\).

The projection body and its polar play an important role in several affine
isoperimetric inequalities, including the Petty projection inequality and Zhang's
reverse inequality. They also arise naturally in the study of convolution bodies
and their limiting behavior. In particular, the level sets of convolution
functions of convex bodies have been studied by Kiener, Schmuckenschl{\"a}ger,
Meyer--Reisner--Schmuckenschl{\"a}ger, and Tsolomitis, among others; see
\cite{kiener1986extremalitat,schmuckenschlager1992distribution,MeyerReisnerSchmuckenschlager1993,
Tsolomitis1997}. These works show that the asymptotic geometry of overlap
functions is deeply connected with projection bodies and affine inequalities.

The convolution-body point of view has also been developed in a series of works
closer to the present setting. Alonso-Gutiérrez, Jiménez and Villa introduced
\(\theta\)-convolution bodies of two convex bodies and proved Brunn--Minkowski
and Zhang-type inequalities for them \cite{AlonsoJimenezVilla2013}. Later,
Alonso-Gutiérrez, González and Jiménez studied \(i\)-th convolution bodies and
obtained volume inequalities for their limiting bodies, including extensions of
Zhang's inequality through their relation with polar projection bodies
\cite{AlonsoGonzalezJimenez2015}. In a functional direction,
Alonso-Gutiérrez, González Merino, Jiménez and Villa extended the convolution-body
construction to log-concave functions in the context of Rogers--Shephard
inequalities \cite{AlonsoGonzalezJimenezVilla2016}. More recently,
Alonso-Gutiérrez and Martín Goñi used Ball's bodies to obtain sharp
Brunn--Minkowski-type estimates for \(\theta\)-convolution bodies in a range of
parameters \cite{AlonsoMartinGoni2024}. Our construction is different: rather
than considering super-level sets of the classical covariogram, we smooth the
covariogram itself by replacing the intersection with a radial mean.

A fundamental feature of the covariogram is that it is generally not smooth at
the origin. Indeed, for a convex body \(K\) with sufficiently regular boundary
one has the first-order expansion
\[
\operatorname{vol}(K)-g_K(tu)
=
t\,h_{\Pi K}(u)+O(t^2),
\qquad t\to0^+.
\]
Thus \(g_K\) has directional derivatives at the origin, but it typically does not
have a Hessian there. The graph of \(g_K\) has a tangent cone at its maximum
rather than a quadratic tangent profile.

Our purpose is to introduce a smooth approximation of the covariogram which
preserves its geometric nature but replaces the intersection \(K\cap(K+x)\) by a
radial mean. Radial mean constructions have appeared naturally in affine convex
geometry, notably in the work of Gardner and Zhang on affine inequalities and
radial mean bodies \cite{GardnerZhang1998}. In the present setting, for \(p>0\)
and \(x\in\operatorname{int}(-K)\), we will define a star body \(K_{x,p}\) by taking
the radial \((-p)\)-mean of \(K\) and \(K+x\). The corresponding generalized
\(p\)-covariogram is
\[
g_{K,p}(x)=\operatorname{vol}(K_{x,p}).
\]
As \(p\to\infty\), the radial \((-p)\)-mean converges to the pointwise minimum
of the two radial functions, and consequently
\[
g_{K,p}(x)\to g_K(x)
\]
for every \(x\in\operatorname{int}(-K)\).

The advantage of this construction is that, for each fixed \(p\), the function
\(g_{K,p}\) is smooth near the origin. Its first derivative vanishes at the
origin, and its Hessian can be computed explicitly. We prove that, for a convex
body \(K\) with \(C^2\) boundary,
\[
H g_{K,p}(0)
=
-\frac{n+p}{4}\mathcal{M}_K.
\]
Thus the infinitesimal quadratic behavior of the smoothed covariogram
has
\(\Pi_2^\circ K\) as its associated unit ball.

This leads to a first local limit theorem. If \(p>0\) is fixed and
\(\delta\to0^+\), then the level sets
\[
L(p,\delta)
=
\{x\in\operatorname{int}(-K):
\operatorname{vol}(K)-g_{K,p}(x)\leq\delta\}
\]
satisfy
\[
\frac{L(p,\delta)\cap B_\rho}{\sqrt{\delta}}
\longrightarrow
\left(\frac{8}{n+p}\right)^{1/2}\Pi_2^\circ K
\]
in the Hausdorff metric, for a sufficiently small neighborhood \(B_\rho\) of the
origin.

On the other hand, if one first lets \(p\to\infty\), the generalized covariogram
converges to the usual covariogram, whose local level sets have \(\Pi^\circ K\) as their asymptotic shape. Thus the two limits
\[
p\to\infty
\qquad\text{and}\qquad
\delta\to0^+
\]
do not commute. The main result of the paper describes the transition between
these regimes when both limits are taken simultaneously.

Let \(\delta=p^\alpha\), with \(\alpha<0\). We prove that three different regimes
occur. If \(\alpha>-1\), then the level sets behave as those of the classical
covariogram and
\[
p^{-\alpha}L(p,p^\alpha)\longrightarrow \Pi^\circ K.
\]
If \(\alpha<-1\), then the quadratic behavior of the smoothed covariogram
dominates and
\[
p^{\frac{1-\alpha}{2}}L(p,p^\alpha)
\longrightarrow
\sqrt{8}\,\Pi_2^\circ K.
\]
The critical case \(\alpha=-1\) gives a genuine transition body. In this case,
\[
pL(p,p^{-1})\longrightarrow \mathcal T_K,
\]
where
\[
\mathcal T_K
=
\{y\in\mathbb R^n:\Phi_K(y)\le1\}
\]
and
\[
\Phi_K(y)
=
\int_{\partial K}
(x\cdot n_K(x))
\log\cosh\left(
\frac{y\cdot n_K(x)}
{2x\cdot n_K(x)}
\right)
\,dS(x).
\]
This body interpolates between the two limiting geometries: near the origin
\(\Phi_K\) has a quadratic expansion determined by \(\mathcal{M}_K\), while at infinity
its linear growth is determined by \(h_{\Pi K}\), since
\[
\lim_{\lambda\to\infty}
\frac{\Phi_K(\lambda y)}{\lambda}
=
\frac12\int_{\partial K}|y\cdot n_K(x)|\,dS(x)
=
h_{\Pi K}(y).
\]

This transition body also has a natural interpretation in the framework of
Orlicz projection bodies. More precisely, the functional defining \(\mathcal T_K\)
is an Orlicz-type boundary integral of the same form as the one appearing in the
definition of the Orlicz projection body of Lutwak--Yang--Zhang \cite{LutwakYangZhang2010OrliczProjection}, for the Orlicz function
\[
\varphi(t)=\log\cosh(t/2).
\]
Thus the critical regime does not produce an isolated new object, but rather
selects a canonical Orlicz projection-type body determined by the smoothing
procedure.

The paper is organized as follows. In Section 2 we establish the notation and
preliminary facts used throughout the manuscript. In Section 3 we prove the
differentiability of \(g_{K,p}\) and compute its first and second derivatives.
This gives a preliminary Hessian formula involving an auxiliary operator \(\mathcal{N}_K\).
In Section 4 we prove the identity
\[
\mathcal{N}_K=(n-1)\mathcal{M}_K,
\]
which yields the explicit Hessian formula. Finally, in Section 5 we study the
level sets of \(g_{K,p}\), proving the fixed-\(p\) limit theorem and the
simultaneous-limit theorem, including the critical transition case.

We now introduce the radial combinations that will be used to define the generalized covariogram. We first recall the relevant terminology for star bodies.

	A set $K\subset \R^n$ is star-shaped with respect to the origin if every ray from the origin intersects $K$ in a line segment. A star body is a star-shaped set whose radial function
	\begin{equation}
		\rho_K(v) = \sup \{ \lambda \ge 0 : \lambda v \in K \}, \qquad v \in S^{n-1},
	\end{equation}
	is continuous and positive on $S^{n-1}$. 
Although the radial function is initially defined on \(S^{n-1}\), we shall use
the same notation for its homogeneous extension to \(\mathbb R^n\setminus\{0\}\),
defined by
\[
\rho_K(z)=\sup\{\lambda\ge0:\lambda z\in K\}.
\]
Thus, if \(z=\lambda v\) with \(v\in S^{n-1}\) and \(\lambda>0\), then
\[
\rho_K(z)=\lambda^{-1}\rho_K(v).
\]
In particular, for every \(T\in GL(n)\),
\[
\rho_{TK}(z)=\rho_K(T^{-1}z),\qquad z\in\mathbb R^n\setminus\{0\}.
\]
	
	\begin{definition}\label{def:radial_lp_sum}
		Let $\beta \in \R \setminus \{0\}$.
		Given two star bodies $K$ and $L$, we define their radial $\beta$-combination, $K \sumTilde[\beta] L$ as the body whose radial function is given by:
		\begin{equation}
			\label{eq_beta_sum}
			\rho_{K \sumTilde[\beta] L}(v)^{\beta} = \frac{\rho_K(v)^{\beta}+ \rho_L(v)^{\beta}}{2}, \qquad \forall v \in S^{n-1}.
		\end{equation}
	\end{definition}
	
	\begin{definition}
		Let $K \subset \R^n$ be a convex body containing the origin in its interior, and let $p >0$. For $x \in \operatorname{int}(-K)$, we write $K_{x,p} = K \sumTilde (K + x)$ and define the generalized $p$-covariogram $g_{K, p}$ as
		\begin{equation}
			g_{K, p}(x) = \vol{K_{x,p}},\qquad x \in \operatorname{int}(-K).
		\end{equation}
	\end{definition}
The function \(g_{K,p}\) is naturally defined on \(\operatorname{int}(-K)\), which
is precisely the set of translations for which \(K+x\) contains the origin in its
interior. Since all the asymptotic results below are local at the origin, this
domain is sufficient for our purposes.

We next compute the Hessian of \(g_{K,p}\) at the origin.
%

	\begin{theorem}\label{thm:second_derivative_covariogram_simplified}
		Let $K \subset \R^n$ be a convex body with $C^2$ boundary which contains the origin in its interior. The Hessian of the function $g_{K,p}$ at the origin is given by
		\begin{equation}
			\label{eq_hessian_gKp}
			H_{g_{K,p}}(0) = -\frac{n+p}{4}\mathcal{M}_K.
		\end{equation}
	\end{theorem}
	Since \(\mathcal M_K\) is positive definite, formula
\eqref{eq_hessian_gKp} shows that \(g_{K,p}\) has a strict local maximum at the
origin for every \(p>0\).

\begin{remark}
Although in this paper we work with \(p>0\), the Hessian computation remains valid if the radial combination is defined by the same formula for
\(p\in\mathbb R\setminus\{0\}\). In this extended range,
\[
H g_{K,p}(0)=-\frac{n+p}{4}\mathcal M_K.
\]
Thus the origin is a strict local maximum whenever \(p>-n\) and \(p\neq0\). At the exceptional value
\(p=-n\), the function \(g_{K,-n}\) is constant. Indeed, in that case the
radial combination is the \(n\)-radial mean and
\[
\rho_{K_{x,-n}}(v)^n
=
\frac{\rho_K(v)^n+\rho_{K+x}(v)^n}{2}.
\]
Therefore
\[
g_{K,-n}(x)
=
\frac12\operatorname{vol}(K)
+
\frac12\operatorname{vol}(K+x)
=
\operatorname{vol}(K).
\]
This extended parameter range will not be used in the sequel.
\end{remark}
	Define the level sets of the generalized covariogram by
	\begin{equation}
		L(p, \delta) = \{x \in \operatorname{int}(-K): \vol{K} - g_{K,p}(x) \leq \delta\}.
	\end{equation}

	The next theorem is an analog of the classical result of convergence of level sets of the covariogram function.
	\begin{theorem}\label{thm:level_sets_convergence}
		Let $K \subset \R^n$ be a convex body with $C^2$ boundary containing the origin in its interior, and assume $p > 0$.
		Then there exists $\rho > 0$ such that:
		\begin{equation}
			\lim_{\delta \to 0^+}\frac{L(p, \delta)\cap B_\rho}{\sqrt{\delta}} = c_{n,p}\;\Pi_2^{\circ} K,
		\end{equation}
		where the convergence is in the Hausdorff metric, and $c_{n,p} = \left(\frac{8}{n+p}\right)^{1/2}$.
	\end{theorem}

	Denote by $\delta(x, p)$, or just $\delta$, the difference $\vol{K} - g_{K,p}(x)$. Theorem \ref{thm:level_sets_convergence}  describes the fixed-\(p\) infinitesimal behavior of the level sets. We
now study the simultaneous regime in which \(p\to\infty\) and \(\delta=\delta(p)\to0\). If we look at the Taylor expansion of the covariogram and its generalized version we have a hint of what happens. For the usual one we have directional derivatives in the origin, while in the generalized version we have that the first derivative is $0$, while its second one explodes as $p \to \infty$. 

	
	If we consider both limits simultaneously, we observe the following interesting phenomenon.
	\begin{theorem}\label{thm:level_sets_coupled_convergence}
		Let \(K\subset \mathbb R^n\) be a convex body with \(C^3\) boundary
which contains the origin in its interior. Let \(\delta=p^\alpha\), where
\(\alpha<0\). Then the following limits hold in the Hausdorff metric:

\begin{enumerate}
\item \label{thm:level_sets_coupled_convergence_case1} If \(\alpha>-1\), then
\[
\lim_{p\to\infty}p^{-\alpha}L(p,p^\alpha)=\Pi^\circ K.
\]

\item \label{thm:level_sets_coupled_convergence_case3}If \(\alpha=-1\), then
\[
\lim_{p\to\infty}pL(p,p^{-1})=\mathcal T_K,
\]
where
\[
\mathcal T_K=\{y\in\mathbb R^n:\Phi_K(y)\le1\}
\]
and
\[
\Phi_K(y)
=
\int_{\partial K}
(x\cdot n_K(x))
\log\cosh\left(
\frac{y\cdot n_K(x)}
{2x\cdot n_K(x)}
\right)
\,dS(x).
\]

\item \label{thm:level_sets_coupled_convergence_case2}If \(\alpha<-1\), then
\[
\lim_{p\to\infty}p^{\frac{1-\alpha}{2}}L(p,p^\alpha)
=
\sqrt 8\,\Pi_2^\circ K.
\]
\end{enumerate}
	\end{theorem}

	Note that the order of the normalization term in this result diverges from the one established in Theorem~\ref{thm:level_sets_convergence}. This happens because the Hessian of $g_{K, p}$ has a dependence on $p$ which is not taken into account in Theorem~\ref{thm:level_sets_convergence}. \\
	\begin{figure}[htbp]
		\centering
		\begin{tikzpicture}[scale=1.5]
		\draw[->, thick] (-3,0) -- (3,0) node[anchor=west] {$x$};
		\draw[->, thick] (0,-0.2) -- (0,3.8);
		
		\coordinate (Peak) at (0,3.15);
		\fill (Peak) circle (1.5pt);
		\node[anchor=south west, font=\small] at (Peak) {$\vol{K}$};
		
		\draw[thick, dashed, gray!80] (-2.5, 0.5) -- (0, 3.15) -- (2.5, 0.5);
		\node[gray!80, anchor=west, font=\footnotesize] at (1.9, 0.4) {$g_K(x)$};
		
		\draw[very thick, blue, domain=-2.6:2.6, samples=100] plot (\x, {3.15 - 0.424*\x*\x});
		\node[blue, anchor=west, font=\footnotesize] at (1.6, 2.1) {$g_{K,p}(x)$};
		
		\coordinate (C1_L) at (-0.84, 2.85);
		\coordinate (C1_R) at (0.84, 2.85);
		
		\draw[thick, orange, <->] (C1_L) -- (C1_R);
		
		\node[orange, anchor=west, font=\footnotesize] at (1.0, 2.95) {Upper Cut: $\delta_1 = p^\alpha$ ($\alpha < -1$)};
		\node[orange, anchor=west, font=\tiny\bfseries] at (0, 2.65) {$\Pi_2^\circ K$};
		
		\draw[gray, <->] (0.15, 3.15) -- (0.15, 2.85) node[midway, right, font=\tiny, text=black] {$\delta_1$};
		
		\coordinate (C3_L) at (-1.91, 1.6);
		\coordinate (C3_R) at (1.91, 1.6);
		
		\draw[thick, violet, <->] (C3_L) -- (C3_R);
		
		\node[violet, anchor=west, font=\footnotesize] at (1.95, 1.7) {Intermediate Cut: $\delta_2 = p^{-1}$ ($\alpha = -1$)};
		\node[violet, anchor=west, font=\tiny\bfseries] at (0, 1.4) {$\mathcal{T}_K$};
		
		\draw[gray, <->] (-0.20, 3.15) -- (-0.20, 1.6) node[midway, left, font=\tiny, text=black] {$\delta_2$};
		
		\coordinate (C2_L) at (-2.5, 0.5);
		\coordinate (C2_R) at (2.5, 0.5);
		
		\draw[thick, red, <->] (C2_L) -- (C2_R);
		
		\node[red, anchor=west, font=\footnotesize] at (2.55, 0.6) {Lower Cut: $\delta_3 = p^\alpha$ ($\alpha > -1$)};
		\node[red, anchor=west, font=\tiny\bfseries] at (0, 0.3) {$\Pi^\circ K$};
		
		\draw[gray, <->] (-0.55, 3.15) -- (-0.55, 0.5) node[midway, left, font=\tiny, text=black] {$\delta_3$};
		
		\end{tikzpicture}
		\caption{Geometric interpretation of the level set cuts on $g_{K,p}$. In this image, the value of $p$ is fixed. Choosing to cut the graph at a distance $\delta = p^{\alpha}$ from the top: if $\alpha < -1$ (upper cut), the level sets converge to $\Pi_2^{\circ} K$; if $\alpha = -1$ (intermediate cut), they converge to the transition body $\mathcal{T}_K$; and if $-1 < \alpha < 0$ (lower cut), they converge to $\Pi^\circ K$, all after an appropriate rescaling.}
	\end{figure}
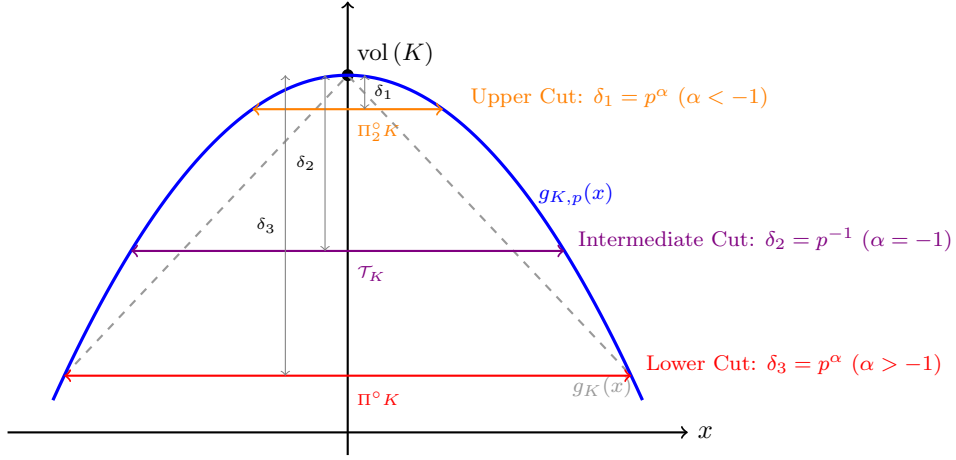

	This paper is organized as follows: In Section~\ref{sec:notation}, we establish the notation and preliminary results that will be used throughout the manuscript. In Section~\ref{sec:derivatives_of_g}, we study the differentiability of the generalized covariogram, establishing the existence of its first and second derivatives and providing their explicit computations; at the end of this section, we establish a preliminary version of Theorem~\ref{thm:second_derivative_covariogram_simplified} with a more complicated expression for $H_{g_{K,p}}(0)$. In Section~\ref{sec:equivalence}, we prove the equivalence between the operator $\mathcal{M}_K$ and an auxiliary operator, which allows us to simplify the expression for the Hessian obtained in the previous section. Finally, in Section~\ref{sec:level_sets}, we prove Theorems~\ref{thm:level_sets_convergence} and \ref{thm:level_sets_coupled_convergence}.

	\FloatBarrier
	\section{Notation and Preliminaries}\label{sec:notation}

	In this paper, we restrict ourselves to the $n$-dimensional Euclidean space $\R^n$. For any $u, v \in \R^n$, we denote by $u\cdot v$ and $u \otimes v$, respectively, the inner product and tensor product of $u$ and $v$. The unit sphere is denoted by $S^{n-1}$ and the volume of the unit ball by $\omega_n$. We use $\mathrm{d}v$ and $\mathrm{d}S$ to denote the $(n-1)$-dimensional Hausdorff measure on the unit sphere and on $\partial K$, respectively. We say that a function $f$ is $\mathcal{O}(g)$ if there exists a constant $C > 0$ such that $|f(x)| \leq C|g(x)|$ for all $x$ in the domain of $f$ and $g$, and we say that $f$ is $o(g)$ if $\lim_{x \to \infty} f(x)/g(x) = 0$.
	
	Let $K$ be a convex body, we define the support function of $K$, denoted by $h_K: \R^n \to \R$, as:
	\begin{equation}
		h_K(u) = \sup \{ x \cdot u : x \in K \}, \quad u \in \R^n.
	\end{equation}

	Let $K$ be a convex body with $C^2$ boundary $M=\partial K$. The shape operator $S_x$ at $x$ is the linear map $S_x: T_xM \to T_xM$ defined by $S_x(v) = \mathrm{d}n_{K(x)}(v)$ for $v \in T_xM$. The second fundamental form $\II_x$ at $x$ is the bilinear form on $T_xM$ defined by $\II_x(u,v) = u \cdot S_x(v)$, we use the positive sign convention consistent with \cite{schneider2013convex}.  We denote by $\nabla$ the Levi-Civita connection on $M$; the covariant derivative of a vector field $Y$ with respect to another vector field $X$ is denoted by $\nabla_X Y$.

	To convert an integral over the boundary of a convex body into an integral over the sphere (or vice versa), the following result is essential. For a proof we refer the reader to Hug \cite{hug1996curvature}.
	\begin{lemma}\label{lem:jacobian_change_of_coordinates}
		Let $K \subset \R^n$ be a convex body containing the origin in its interior, and let $f: S^{n-1} \to \partial K$ be the radial map defined by $f(v) = \rho_K(v)v$. Then $f$ is differentiable for almost all $v \in S^{n-1}$ with respect to the $(n-1)$-dimensional Hausdorff measure, and its Jacobian determinant is given by:
		\begin{equation}
			Jf(v) = \frac{\rho_K(v)^{n-1}}{v\cdot n_K\left(\rho_K(v)v\right)}
		\end{equation}
		In other words, for every continuous function $\phi$ on the boundary $\partial K$,
		\begin{equation}
			\int_{\partial K} \phi(x) \mathrm{d}S(x) = \int_{S^{n-1}} \phi\left(\rho_K(v)v\right) \frac{\rho_K(v)^{n-1}}{v\cdot n_K\left(\rho_K(v)v\right)} \mathrm{d}S(v).
		\end{equation}
	\end{lemma}

	\section{The First and Second Derivatives of \texorpdfstring{$g_{K,p}$}{g\_\{K,p\}}}
	\label{sec:derivatives_of_g}
	\begin{lemma}\label{lem:smooth_boundary_implies_smooth_rho}
		Let $K$ be a $C^k$ convex body containing the origin in its interior. Then the function $(x, v) \mapsto \rho_{K_{x, p}}(v)$ is of class $C^k$ in $\operatorname{int} (-K) \times S^{n-1}$. Moreover those derivatives are uniformly bounded on any compact set $L\subset \operatorname{int} (-K) \times S^{n-1}$.
	\end{lemma}
	\begin{proof}
		It suffices to show that the function $(x, v) \mapsto \rho_{K+x}(v)$ is of class $C^k$, since the function
		\begin{equation}
			y \mapsto \left(\frac{y^{-p} + a^{-p}}{2}\right)^{-1/p}
		\end{equation}
		is in $C^\infty((0,\infty))$, for $a > 0$.
		
		Let $p(v_0) = \rho_K(v_0)v_0$. Since $K$ is of class $C^k$, there exists a $C^k$ defining function $\Phi$ in a neighborhood $U$ of $p(v_0)$ such that
		\begin{equation}
			\partial K \cap U = \{y\in U:\Phi(y) = 0\},\qquad \nabla\Phi(p(v_0))\neq0.
		\end{equation}
		Consider $F(x,v,t) = \Phi(tv-x)$. Note that $F(0, v_0, \rho_K(v_0)) = \Phi(\rho_K(v_0)v_0) = 0$. Moreover, since $\nabla\Phi(p(v_0))$ is normal to $\partial K$ at $p(v_0)$ and the origin is in the interior, the following transversality condition holds:
		\begin{equation}
			\partial_t F(0,v_0,\rho_K(v_0)) = \nabla\Phi(p(v_0))\cdot v_0\neq 0.
		\end{equation}
		By the implicit function theorem, there exists a neighborhood $V\subset \R^n \times S^{n-1}$ of $(0, v_0)$ and a unique $C^k$ function $g:V\to\R$ such that $g(0, v_0) = \rho_K(v_0)$ and $F(x, v, g(x, v)) = 0$ for all $(x, v)\in V$. By uniqueness, $\rho_{K+x}(v) = g(x, v)$ near $(0, v_0)$, thus $(x,v)\mapsto\rho_{K+x}(v)$ is $C^k$ near $(0, v_0)$.
		
		By translating $K$, it is enough to show the regularity for $x$ near $0$. Let $x_0$ be such that $0 \in \operatorname{int}(K+x_0)$. Since $K+x_0$ is a convex body containing the origin in its interior, by the previous argument the function
		\begin{equation}
			f(x, v) = \rho_{(K+x_0)+x}(v)
		\end{equation}
		is $C^k$ in a neighborhood of $(0, v_0)$. Now, consider the translation map $\psi(x, v) = (x-x_0, v)$. The composition
		\begin{equation}
			(f \circ \psi)(x, v) = f(x-x_0, v) = \rho_{K+x_0+(x-x_0)}(v) = \rho_{K+x}(v)
		\end{equation}
		is a composition of $C^k$ functions, and therefore is $C^k$ in a neighborhood of $x_0$, as required.

		The uniform boundedness of the derivatives follows directly from the compactness of $L$ and the continuity of the derivatives.
	\end{proof}

	\begin{proposition}\label{prop:covariogram_is_Ck}
		Let $K$ be a $C^k$ convex body containing the origin in its interior.
		Then the function $x \mapsto \vol{K_{x, p}}$ is of class $C^k$ in $\operatorname{int}(-K)$.
	\end{proposition}
	
	\begin{proof}
		From Lemma \ref{lem:smooth_boundary_implies_smooth_rho}, the function $x \mapsto \rho_{K_{x, p}}(v)$ is of class $C^k$, and all its derivatives up to order $k$ are uniformly bounded in a neighborhood of $x_0\in \operatorname{int}(-K)$, with a bound independent of $v \in S^{n-1}$.
		
		Recall that the volume of $K_{x, p}$ can be expressed in terms of the radial function as
		\begin{equation}
			\vol{K_{x, p}} = \frac{1}{n}\int_{S^{n-1}} \rho_{K_{x, p}}(v)^n \, \mathrm{d}v.
		\end{equation}
		Since $\rho_{K_{x, p}}(v)$ and its derivatives depend smoothly on $x$ and are uniformly bounded in $v$, we may differentiate under the integral sign by the dominated convergence theorem. Hence, for every multi-index $\alpha$ with $|\alpha| \leq k$ and every $x \in \operatorname{int} (-K)$,
		\begin{equation}
			\partial^\alpha_x \vol{K_{x, p}} = \frac{1}{n}\int_{S^{n-1}} \partial^\alpha_x \big(\rho_{K_{x, p}}(v)^n\big) \, \mathrm{d}v. \qedhere
		\end{equation}
	\end{proof}
	
	\begin{lemma}\label{lem:first_derivative_translated}
		Let $K$ be a convex body with $C^1$ boundary, let $v\in S^{n-1}$, and let $x_0(v) = \rho_K(v)\,v$. Then:
		$\left. \frac{\partial}{\partial t}\rho_{K+tu}(v)\right|_{t=0}= \frac{u\cdot n_K(x_0(v))}{v\cdot n_K(x_0(v))}$
	\end{lemma}
	
	\begin{proof}
		Let us define $\varphi(t) = \rho_{K+tu}(v)$.
		\begin{align*}
			\varphi(t) &= \max\{s\geq 0: sv\in K+tu\}\\ &= \max\{s\geq 0: sv-tu \in K\}
		\end{align*}
		
		Therefore $sv-tu \in \partial K$ if $s=\varphi(t)$. By taking the derivative of equation $h_K(n_K(x_0(v))) = (sv-tu)\cdot n_K(x_0(v))$ and using the continuity of $n_K(x_0(v))$
		we arrive at $0 = \varphi'(0)( v\cdot n_K(x_0(v))) - u\cdot n_K(x_0(v)) $.
	\end{proof}
	\begin{corollary}\label{cor:first_derivative_Kt}
		Let $K$ be a convex body with $C^1$ boundary, let $v\in S^{n-1}$, and let $x_0(v) = \rho_K(v)\,v$. Then:

		$\left. \frac{\partial}{\partial t}\rho_{K_{tu,p}}(v)\right|_{t=0}= \frac{ u\cdot n_K(x_0(v))}{2v\cdot n_K(x_0(v))}$ for every $u\in S^{n-1}$
	\end{corollary}
	\begin{proof}
		\begin{align*}
			\left.\frac{\partial}{\partial t}\rho_{K_{tu,p}}(v)\right\rvert_{t=0} &= \left. \frac{\partial}{\partial t}\left( \frac{\rho_{K}(v)^{-p} + \rho_{K+tu}(v)^{-p}}{2}\right)^{-1/p}\right\rvert_{t=0} \\
			&= \left. -\frac{1}{p}\left(-\frac{p}{2}\frac{\partial}{\partial t}\rho_{K+tu}(v)\rho_{K+tu}(v)^{-p-1}\right)\rho_{K_{tu,p}}(v)^{1+p}\right\rvert_{t=0}\\
			&= \frac{u \cdot n_K(x_0(v))}{2v \cdot n_K(x_0(v))} \qedhere
		\end{align*}
	\end{proof}

	For each $x\in \partial K$, define the operator $A_x$ by:
	\begin{equation}
		A_x = I - \frac{x \otimes n_K(x)}{x\cdot n_K(x)}.
	\end{equation} 
	Define $\mathcal{N}_K$ by the following integral:
	\begin{equation}
		\mathcal{N}_K = \int_{\partial K} A_x^t S_x A_x \, \mathrm{d}S(x).
	\end{equation}

	\begin{lemma}\label{lem:second_derivative_translated}
		Let $K \subset \R^n$ be a convex body with $C^2$ boundary, let $v \in S^{n-1}$, and let $x_0(v) = \rho_K(v)\,v$.
		Then:
		\begin{equation}
			\left.\frac{\partial^2}{\partial t^2}\rho_{K+tu}(v)\right|_{t=0}
			= -\frac{u^t Q_{x_0(v)} u}{v\cdot n_K(x_0(v))},
		\end{equation}
		where $S_{x_0(v)}$ is the shape operator at $x_0(v)$ and $Q_{x_0(v)} = A_{x_0(v)}^t S_{x_0(v)} A_{x_0(v)}$ is a positive semidefinite symmetric matrix.
	\end{lemma}
	
	\begin{proof}
Define \(x(t)=\varphi(t)v-tu\). We have
\[
x'(0)=\varphi'(0)v-u=-A_{x_0(v)}u.
\]
Since \(x(t)\in \partial K\), we have
\[
x'(t)\cdot n_K(x(t))=0.
\]
Differentiating this identity with respect to \(t\), and evaluating at \(t=0\),
we obtain
\[
0
=
\left.\frac{d}{dt}
\big(x'(t)\cdot n_K(x(t))\big)\right|_{t=0}.
\]
By the product rule and the chain rule,
\[
0
=
x''(0)\cdot n_K(x_0(v))
+
x'(0)\cdot d n_K{}_{x_0(v)}(x'(0)).
\]
Using the convention \(S_x=d n_K{}_x\), this becomes
\[
0
=
x''(0)\cdot n_K(x_0(v))
+
x'(0)\cdot S_{x_0(v)}x'(0).
\]
Therefore
\[
x''(0)\cdot n_K(x_0(v))
=
-\mathrm{II}_{x_0(v)}(x'(0),x'(0)).
\]
Hence
\[
\begin{aligned}
(v\cdot n_K(x_0(v)))\varphi''(0)
&=
\varphi''(0)v\cdot n_K(x_0(v))  \\
&=
x''(0)\cdot n_K(x_0(v)) \\
&=
-\mathrm{II}_{x_0(v)}(x'(0),x'(0)) \\
&=
-\mathrm{II}_{x_0(v)}(-A_{x_0(v)}u,-A_{x_0(v)}u) \\
&=
-\mathrm{II}_{x_0(v)}(A_{x_0(v)}u,A_{x_0(v)}u) \\
&=
-(A_{x_0(v)}u)\cdot S_{x_0(v)}(A_{x_0(v)}u) \\
&=
-u^tA_{x_0(v)}^tS_{x_0(v)}A_{x_0(v)}u \\
&=
-u^tQ_{x_0(v)}u .
\end{aligned}
\]
Thus
\[
\varphi''(0)
=
-\frac{u^tQ_{x_0(v)}u}{v\cdot n_K(x_0(v))}.
\]
\end{proof}
	\begin{corollary}\label{cor:second_derivative_Kt}
		Let $K \subset \R^n$ be a convex body with $C^2$ boundary, let $v \in S^{n-1}$, and let $x_0(v) = \rho_K(v)\,v$.
		Then:
		\begin{equation}
			\left.\frac{\partial^2}{\partial t^2}\rho_{K_{tu,p}}(v)\right|_{t=0}
			= -\frac{1}{2}\frac{u^t Q_{x_0(v)} u}{v\cdot n_K(x_0(v))}- \frac{p+1}{4}\frac{(u\cdot n_K(x_0(v)))^2}{(v\cdot n_K(x_0(v)))^2} \rho_{K}^{-1}(v)
		\end{equation}
	\end{corollary}
	\begin{proof}
		\begin{align*}
			\left. \frac{\partial^2}{\partial t^2}\rho_{K_{tu,p}}(v)\right\rvert_{t=0}
			&= \frac{1}{2}\frac{\partial^2}{\partial t^2}\rho_{K+tu}(v) - \frac{p+1}{4}\left(\frac{\partial}{\partial t}\rho_{K+tu}(v)\right)^2 \rho_{K+tu}^{-1}(v)\\
			&= -\frac{1}{2}\frac{ u^t Q_{x_0(v)} u}{v\cdot n_K(x_0(v))} -\frac{p+1}{4}\frac{(u \cdot n_K(x_0(v)))^2}{(v\cdot n_K(x_0(v)))^2} \rho_{K}^{-1}(v) \qedhere
		\end{align*}
	\end{proof}
	
	\begin{proposition}\label{prop:first_derivative_covariogram}
		Let $K$ be a convex body with $C^1$ boundary. Then $\nabla g_{K,p}(0) = 0$.
	\end{proposition}
	\begin{proof}
		Using Corollary \ref{cor:first_derivative_Kt}, Lemmas \ref{lem:jacobian_change_of_coordinates}, \ref{lem:smooth_boundary_implies_smooth_rho} and the divergence theorem with the constant vector field u,
		\begin{align*}
			\left. \frac{\partial}{\partial t} g_{K,p}(tu)\right\rvert_{t=0} &= \left. \frac{\partial}{\partial t}\left(\frac{1}{n}\int_{S^{n-1}}\rho_{K_{tu,p}}(v)^n\mathrm{d}v \right) \right\rvert_{t=0} \\
			&= \left. \int_{S^{n-1}}\frac{\partial}{\partial t}\rho_{K_{tu,p}}(v)\rho_{K_{tu,p}}(v)^{n-1}\mathrm{d}v \right\rvert_{t=0}\\
			&= \frac{1}{2}\int_{S^{n-1}}\frac{u\cdot n_K(x_0(v))}{v \cdot n_K(x_0(v))}\rho_K(v)^{n-1}\mathrm{d}v\\
			&= \frac{1}{2}\int_{\partial K} u\cdot n_K(x) \, \mathrm{d}S(x)\\
			&= \frac{1}{2}\int_{K} \operatorname{div} u \; dx\\
			&= 0 \qedhere
		\end{align*}
	\end{proof}
	
	The following result gives an explicit formula for the Hessian of $g_{K,p}$ at the origin, which unlike Theorem \ref{thm:second_derivative_covariogram_simplified}, depends on $\mathcal{N}_K$. The latter will follow automatically after we establish that $\mathcal{N}_K = (n-1)\mathcal{M}_K$ in the next section.

	\begin{theorem}\label{thm:second_derivative_covariogram}
		Let $K \subset \R^n$ be a convex body with $C^2$ boundary which contains the origin in its interior. The Hessian of the function $g_{K,p}$ at the origin is given by
		\begin{equation}
			Hg_{K,p}(0) = \frac{n-p-2}{4}\mathcal{M}_K - \frac{1}{2}\mathcal{N}_K.
		\end{equation}
	\end{theorem}
	\begin{proof}
		Using Corollaries \ref{cor:first_derivative_Kt}, \ref{cor:second_derivative_Kt} and Lemmas \ref{lem:jacobian_change_of_coordinates}, \ref{lem:smooth_boundary_implies_smooth_rho} we obtain (writing $x_0(v) = \rho_K(v)v$ for the contact point at each $v$):
		\begin{align*}
			& \left. \frac{\partial^2}{\partial t^2} g_{K,p}(tu)\right\rvert_{t=0} = \left. \frac{\partial^2}{\partial t^2}\left(\frac{1}{n}\int_{S^{n-1}}\rho_{K_{tu,p}}(v)^n\mathrm{d}v \right) \right\rvert_{t=0} \\
			&= \left. \int_{S^{n-1}} (n-1)\left(\frac{\partial}{\partial t}\rho_{K_{tu,p}}(v)\right)^2\rho_K(v)^{n-2} + \left(\frac{\partial^2}{\partial t^2}\rho_{K_{tu,p}}(v)\right)\rho_K(v)^{n-1})\mathrm{d}v \right\rvert_{t=0}\\
			&=\int_{S^{n-1}} \frac{n-p-2}{4} \frac{(u \cdot n_K(x_0(v)))^2}{(v\cdot n_K(x_0(v)))^2}\rho_K(v)^{n-2}-\frac{1}{2}\frac{u^t Q_{x_0(v)} u}{v\cdot n_K(x_0(v))}\rho_K(v)^{n-1} \mathrm{d}v\\
			&= \frac{n-p-2}{4} \int_{\partial K} \frac{(u \cdot n_K(x))^2}{h_K(n_K(x))}\mathrm{d}S-\frac{1}{2} u^t\left(\int_{S^{n-1}} \frac{Q_{x_0(v)}}{v\cdot n_K(x_0(v))} \rho_K(v)^{n-1} \mathrm{d}v \right)u\\
			&= \frac{n-p-2}{4} u^t \mathcal{M}_K u-\frac{1}{2} u^t \mathcal{N}_K u \qedhere
		\end{align*}
	\end{proof}
	
	\section{Equivalence of the operators}\label{sec:equivalence}

		\begin{lemma}\label{lem:invariance_of_operators}
			Let $T\in Gl(n)$, and assume that $K$ is a convex body with $C^2$ boundary containing the origin in its interior. Then: 
			\begin{align*}
				H_{g_{TK,p}} &=|\det T|\,T^{-t}H_{g_{K,p}}T^{-1};\\
				\mathcal{M}_{TK} &=|\det T|\,T^{-t}\mathcal{M}_KT^{-1};\\
				\mathcal{N}_{TK} &=|\det T|\,T^{-t}\mathcal{N}_KT^{-1},
			\end{align*}
			where we use $H_{g_{TK,p}}, H_{g_{K,p}}$ to denote respectively the Hessian of $g_{TK,p}$ and $g_{K,p}$ at the origin.
		\end{lemma}
		\begin{proof}
			Since $\rho_{TK}(v) = \rho_K(T^{-1}v)$, we have:
			\begin{align*}
				\rho_{TK \sumTilde (TK + x)}^{-p}(v)
				&= \frac{\rho_{TK}^{-p}(v) + \rho_{TK + x}^{-p}(v) }{2}\\
				&= \frac{\rho_{K}^{-p}(T^{-1}v) + \rho_{K + T^{-1}x}^{-p}(T^{-1}v) }{2}\\
				&= \rho_{K \sumTilde (K + T^{-1}x)}^{-p}(T^{-1}v) \\
				&= \rho_{T(K \sumTilde (K + T^{-1}x))}^{-p}(v),
			\end{align*}
			so $TK \sumTilde (TK + x) = T\bigl[K \sumTilde (K + T^{-1}x)\bigr]$, and therefore:
			\begin{equation*}
				g_{TK,p}(x) = |\det T|\,g_{K,p}(T^{-1}x).
			\end{equation*}
			Differentiating twice in the direction $u$ and applying the chain rule gives
			\begin{equation*}
				D^2 g_{TK,p}(0)[u,u]
				= |\det T|\, D^2 g_{K,p}(0)[T^{-1}u,\, T^{-1}u].
			\end{equation*}
			Therefore, for every $u \in \R^n$:
			\begin{equation*}
				u^t H_{g_{TK,p}}\, u
				= |\det T|\,(T^{-1}u)^t H_{g_{K,p}}\,(T^{-1}u)
				= |\det T|\, u^t\bigl(T^{-t}H_{g_{K,p}}\,T^{-1}\bigr)u.
			\end{equation*}
			Since this holds for all $u$, we conclude
			\begin{equation}
				\label{eq:H_TK_transform}
				H_{g_{TK,p}} = |\det T|\,T^{-t}H_{g_{K,p}}T^{-1}.
			\end{equation}
			For the other two operators, we can use the relation in \eqref{eq:H_TK_transform} and the expression given by Theorem \ref{thm:second_derivative_covariogram} to write:
			\begin{align*}
				\frac{n-p-2}{4}\mathcal{M}_{TK}-\frac{1}{2} \mathcal{N}_{TK} &= H_{g_{TK,p}}\\
				&=|\det T|\,T^{-t}H_{g_{K,p}}T^{-1}\\
				&=|\det T|\,T^{-t}\left(\frac{n-p-2}{4}\mathcal{M}_{K}-\frac{1}{2} \mathcal{N}_{K}\right)T^{-1}.
			\end{align*}
			By taking the derivative with respect to $p$ on both sides we conclude:
			\begin{equation}
				\label{eq:M_TK_transform}
				\mathcal{M}_{TK}=|\det T|\,T^{-t}\mathcal{M}_KT^{-1}.
			\end{equation}
			Now substitute \eqref{eq:M_TK_transform} into \eqref{eq:H_TK_transform} to conclude that:
			\begin{equation}
				\mathcal{N}_{TK}=|\det T|\,T^{-t}\mathcal{N}_KT^{-1}.\qedhere
			\end{equation}
		\end{proof}
			\begin{lemma}\label{lem:equivalence_of_traces}
			Let $K \subset \R^n$ be a convex body with $C^2$ boundary containing the origin in its interior. Then: 
			\begin{equation}
				\operatorname{tr}(\mathcal{N}_K) = (n-1)\,\operatorname{tr}(\mathcal{M}_K).
			\end{equation}
		\end{lemma}

		\begin{proof}
			By definition, $\operatorname{tr}(\mathcal{N}_K) = \int_{\partial K} \operatorname{tr}(A_x^t S_x A_x)\, \mathrm{d}S(x)$.

			\medskip\noindent\textit{Step 1: Computation of $\operatorname{tr}(A_x^t S_x A_x)$.}
			At each $x \in \partial K$, choose an orthonormal basis $\{e_1, \dots, e_{n-1}\}$ for $T_x\partial K$. Since $A_x$ acts as the identity on $T_x\partial K$ and maps $n_K(x)$ to $-x_T/(x \cdot n_K(x))$:
			\begin{align*}
				\operatorname{tr}(A_x^t S_x A_x)
				&= \sum_{i=1}^{n-1} (S_x A_x e_i) \cdot (A_x e_i) + (S_x A_x n_K(x)) \cdot (A_x n_K(x)) \\
				&= \sum_{i=1}^{n-1} (S_x e_i) \cdot e_i + \frac{(S_x x_T) \cdot x_T}{(x \cdot n_K(x))^2} \\
				&= \operatorname{tr}(S_x) + \frac{\II(x_T, x_T)}{(x \cdot n_K(x))^2}.
			\end{align*}
			Integrating over $\partial K$ gives:
			\begin{equation}
				\operatorname{tr}(\mathcal{N}_K) = \int_{\partial K} \operatorname{tr}(S_x)\, \mathrm{d}S(x) + \int_{\partial K} \frac{\II(x_T, x_T)}{(x \cdot n_K(x))^2}\, \mathrm{d}S(x).
			\end{equation}

			\medskip\noindent\textit{Step 2: $\II(x_T, x_T) = x_T \cdot \nabla_{\partial K}(h_K \circ n_K)(x)$.}
			\begin{align*}
				\II(x_T, x_T) &= x_T \cdot S_x(x_T)\\
				&= x_T \cdot \sum_{i=1}^{n-1} (S_x x_T \cdot e_i) e_i\\
				&= x_T \cdot \sum_{i=1}^{n-1} (x_T \cdot S_x e_i) e_i\\
				&= x_T \cdot \sum_{i=1}^{n-1} (x \cdot S_x e_i) e_i\\
				&= x_T \cdot \sum_{i=1}^{n-1} \left(d(h_K)_{n_K(x)} dn_{K_x}(e_i) \right)e_i\\
				&= x_T \cdot \nabla_{\partial K}(h_K \circ n_K)(x).
			\end{align*}

			\medskip\noindent\textit{Step 3: Evaluating the integral via integration by parts.}

			Using Step~2 and $x \cdot n_K(x) = h_K(n_K(x))$ for $x \in \partial K$, followed by an integration by parts, we have:
			\begin{align*}
				\int_{\partial K} \frac{\II(x_T, x_T)}{(x \cdot n_K(x))^2}\, \mathrm{d}S(x)
				&= \int_{\partial K} \frac{\nabla_{\partial K}(h_K \circ n_K)(x) \cdot x}{h_K(n_K(x))^2}\, \mathrm{d}S(x)\\
				&= -\int_{\partial K} \nabla_{\partial K}\left(\frac{1}{h_K(n_K(x))}\right) \cdot x \; \mathrm{d}S(x)\\
				&= \int_{\partial K} \frac{\operatorname{div}_{\partial K}(x)}{h_K(n_K(x))}\, \mathrm{d}S(x).
			\end{align*}

			Note that when restricted to $\partial K$, we can change its Riemannian connection by the Riemannian connection of the ambient space to compute the surface divergence of $x_T$:
			\begin{align*}
				\operatorname{div}_{\partial K}(x) &= \sum_{i=1}^{n-1} \nabla_{e_i}x_T \cdot e_i \\
				&= \sum_{i=1}^{n-1} \nabla_{e_i}x \cdot e_i - \sum_{i=1}^{n-1} \nabla_{e_i}\left((x\cdot n) n\right) \cdot e_i \\
				&= \sum_{i=1}^{n-1} \nabla_{e_i}x \cdot e_i - h_K(n_K(x))\sum_{i=1}^{n-1} \nabla_{e_i} n \cdot e_i \\
				&= (n-1) - h_K(n_K(x)) \operatorname{tr}(S_x).\\
			\end{align*}

			Now substituting back into the integral:
			\begin{align*}
				\operatorname{tr}(\mathcal{N}_K) &= \int_{\partial K} \operatorname{tr}(S_x)\, \mathrm{d}S(x) + \int_{\partial K} \frac{\operatorname{div}_{\partial K}(x)}{h_K(n_K(x))}\, \mathrm{d}S(x)\\
				&= (n-1)\int_{\partial K} \frac{\mathrm{d}S(x)}{h_K(n_K(x))}\\
				&= (n-1)\int_{\partial K} \frac{|n_K(x)|^2}{h_K(n_K(x))}\, \mathrm{d}S(x)\\
				&= (n-1)\operatorname{tr}(\mathcal{M}_K).\qedhere
			\end{align*}
		\end{proof}

		\begin{theorem}\label{thm:N_equals_n-1_M}
			Let $K \subset \R^n$ be a convex body with $C^2$ boundary which contains the origin in its interior. Then:
			\begin{equation}
				\mathcal{N}_K = (n-1)\,\mathcal{M}_K.
			\end{equation}
		\end{theorem}

		\begin{proof}
			Let $T \in GL(n)$. Apply Lemma~\ref{lem:equivalence_of_traces} with $K$ replaced by $TK$ to obtain:
			\begin{equation}
				\operatorname{tr}(\mathcal{N}_{TK}) = (n-1) \operatorname{tr}(\mathcal{M}_{TK}).
			\end{equation}
			Replace $\mathcal{N}_{TK}$ and $\mathcal{M}_{TK}$ using Lemma~\ref{lem:invariance_of_operators} to get:
			\begin{align*}
				\operatorname{tr}(\mathcal{N}_{TK}) &= |\det T| \operatorname{tr}(\mathcal{N}_K(T^t T)^{-1}), \\
				\operatorname{tr}(\mathcal{M}_{TK}) &= |\det T| \operatorname{tr}(\mathcal{M}_K(T^t T)^{-1}).
			\end{align*}

			After subtracting the two equations, we arrive at:
			\begin{equation}
				\operatorname{tr}((\mathcal{N}_K - (n-1)\mathcal{M}_K)(T^t T)^{-1}) = 0.
			\end{equation}

			Set $A := \mathcal{N}_K - (n-1)\mathcal{M}_K$ and $B := (T^t T)^{-1}$. As $T$ varies over $GL(n)$, $B$ ranges over the set $\mathcal{P}$ of symmetric positive definite matrices, so:
			\begin{equation}
				\label{eq:trace_zero_for_all_B}
				\operatorname{tr}(AB) = 0 \quad \forall B \in \mathcal{P}.
			\end{equation}

			Since $A$ is symmetric, we have that $I + \varepsilon A \in \mathcal{P}$ for $\varepsilon$ sufficiently small. Therefore:
			\begin{equation}
				0 = \operatorname{tr}(A(I + \varepsilon A)) = \operatorname{tr}(A) + \varepsilon \, \operatorname{tr}(A^2) = \varepsilon \|A\|_F^2,
			\end{equation} 
			where $\|\cdot\|_F$ denotes the Frobenius norm and the trace of $A$ is zero because we can take $B=I$ in equation~\eqref{eq:trace_zero_for_all_B}. Hence we conclude that $A = 0$.
		\end{proof}
		\begin{proof}[Proof of Theorem~\ref{thm:second_derivative_covariogram_simplified}]
			Using Theorem~\ref{thm:second_derivative_covariogram} and Theorem~\ref{thm:N_equals_n-1_M}, we have the following expression for the Hessian of $g_{K,p}$ at the origin:
			\begin{equation}
				\frac{n-p-2}{4}\mathcal{M}_K - \frac{n-1}{2}\mathcal{M}_K = -\frac{n+p}{4}\mathcal{M}_K.\qedhere
			\end{equation}
		\end{proof}
	\section{The level sets of \texorpdfstring{$g_{K,p}$}{g\_\{K,p\}}}\label{sec:level_sets}
	
	\begin{lemma}\label{lem:convergence_of_level_sets}
		Let $f:\R^n \to \R$ be a $C^2$ function such that $f(0) =k$ and $\nabla f(0) =0$. Assume that the Hessian of $f$ at the origin, denoted by $H$, is negative definite. Then there exists $\rho > 0$ such that, defining the sets
		\begin{equation}
			S_{\delta} = \frac{\{\,x \in B_\rho : f(x) \geq k - \delta\,\}}{\sqrt{\delta}},\qquad S = \{\,y \in \R^n : \tfrac{1}{2} y^t H y \ge -1\,\}.
		\end{equation}
		We have that $S_\delta \to S$ in the Hausdorff metric as $\delta \to 0^+$.
	\end{lemma}

	\begin{proof}
		By Taylor's theorem, we have:
		\begin{equation}
			f(x) = k + \frac{1}{2} x^t H x + r(x),
			\qquad \text{where } \lim_{x\to 0}\frac{r(x)}{|x|^2} = 0.
		\end{equation}
		
		Since $H$ is negative definite, there exists a constant $c > 0$ such that $y^t H y \leq -c|y|^2$ for all $y \in \R^n$.
		Choose $\rho > 0$ such that, when restricted to $B_\rho$, $f$ attains its unique maximum at the origin, and $|r(x)|/|x|^2 < c/4$ for all $x \in B_\rho$.
		
		Substituting $x = \sqrt{\delta}\,y$ into the Taylor expansion and dividing by $\delta$ yields
		\begin{equation}
			\frac{f(\sqrt{\delta}\,y) - k}{\delta}
			= \frac{1}{2} y^t H y + \frac{r(\sqrt{\delta}\,y)}{\delta}.
		\end{equation}

		Furthermore, we can rewrite $S_\delta$ as
		\begin{equation}
			S_\delta = \left\{\,y \in B_{\rho/\sqrt{\delta}} : \frac{f(\sqrt{\delta}\, y) - k}{\delta} \geq -1\,\right\}.
		\end{equation}
		
		Choose $R > \sqrt{8/c}$ large enough so that $S \subset B_R$, and define $\delta_0 = \rho^2/R^2$. Consequently, for all $\delta < \delta_0$, we have $B_R \subset B_{\rho/\sqrt{\delta}}$.
		
		Fix $\varepsilon>0$. We divide the analysis of $S_\delta$ into two regions:
		
		\medskip\noindent\textit{Case 1: $y \in B_R$.}
		Since $r(x) = o(|x|^2)$ as $x \to 0$, there exists $\delta_1 \in (0, \delta_0)$ such that for all $\delta < \delta_1$ and $y \in B_R$:
		\begin{equation}
			\left|\frac{r(\sqrt{\delta}\,y)}{\delta}\right| = \frac{|r(\sqrt{\delta}\,y)|}{|\sqrt{\delta}y|^2} |y|^2 < \varepsilon.
		\end{equation}
		Thus, in this region, we obtain:
		\begin{equation}
			\frac12 y^t H y - \varepsilon < \frac{f(\sqrt{\delta}\,y)-k}{\delta} < \frac12 y^t H y + \varepsilon.
		\end{equation}
		
		\medskip\noindent\textit{Case 2: $y \in B_{\rho/\sqrt{\delta}} \setminus B_R$.}
		We have the following estimate:
		\begin{align*}
			\frac{f(\sqrt{\delta}\,y)-k}{\delta} &\leq \frac12 y^t H y + \frac{|r(\sqrt{\delta}\,y)|}{\delta}\\
			&\leq -\frac{c}{2}|y|^2 + \frac{|r(\sqrt{\delta}\,y)|}{|\sqrt{\delta}y|^2} |y|^2 \\
			&\leq -\frac{c}{2}|y|^2 + \frac{c}{4} |y|^2\\
			&= -\frac{c}{4} |y|^2.
		\end{align*}
		Since $|y| > R > \sqrt{8/c}$, it follows that $-|y|^2 < -8/c$. This implies the previous estimate is strictly less than $-1$, hence $y \notin S_\delta$.

		Therefore, $S_\delta \subset B_R$. This fact, combined with the bounds from Case 1, yields the inclusions
		\begin{align*}
			\{y \in B_R : \tfrac12 y^t H y \ge -1 + \varepsilon \}
			&\subset S_\delta \subset
			\{y \in B_R : \tfrac12 y^t H y \ge -1 - \varepsilon \},\\
			(1-\varepsilon)^{1/2} S &\subset S_\delta \subset (1+\varepsilon)^{1/2} S.
		\end{align*}
		Letting $\delta \to 0$ (and consequently $\varepsilon \to 0$), the conclusion follows.
	\end{proof}
	
	\begin{proof}[Proof of Theorem~\ref{thm:level_sets_convergence}]
		By Propositions~\ref{prop:covariogram_is_Ck} and~\ref{prop:first_derivative_covariogram}, $g_{K,p}$ is $C^2$ and $\nabla g_{K,p}(0) = 0$. By Theorem~\ref{thm:N_equals_n-1_M} and Theorem~\ref{thm:second_derivative_covariogram}, $H_{g_{K,p}} = -\frac{n+p}{4}\mathcal{M}_K$. Since $\mathcal{M}_K$ is positive definite (as $h_{\Pi_2 K}^2(u) = u^t\mathcal{M}_K u > 0$ for all $u\neq 0$) and $p > 0$, the Hessian $H_{g_{K,p}}$ is negative definite. Therefore the convergence of $S_{\delta,p}$ to $c_{n,p}\;\Pi_2^{\circ} K$ follows immediately from Lemma~\ref{lem:convergence_of_level_sets} applied to $f = g_{K,p}$ and $k = g_{K,p}(0)$, with $H = H_{g_{K,p}}$.
	\end{proof}

		
		\begin{proof}[Proof of Theorem~\ref{thm:level_sets_coupled_convergence}]
	We write
		\[
		\Delta_p(x)=\operatorname{vol}(K)-g_{K,p}(x),
		\qquad
		\Delta_\infty(x)=\operatorname{vol}(K)-g_K(x).
		\]

We first recall a Taylor estimate that will be used in the cases
\(\alpha=-1\) and \(\alpha<-1\). Since \(K\) has \(C^3\) boundary, the proof of
Proposition~~\ref{prop:covariogram_is_Ck} and the computations in Section~\ref{sec:derivatives_of_g} imply that, in a sufficiently
small ball \(B_\rho\) around the origin,
		\begin{equation}
    			\label{eq:taylor_delta_p}
   			 \Delta_p(x)=\frac{n+p}{8}x^t\mathcal{M}_Kx+O(p^2|x|^3).
		\end{equation}
where the constant in the \(O(\cdot)\) term is independent of \(p>0\) and
\(x\in B_\rho\). Indeed, the third derivatives of \(g_{K,p}\) are bounded by
\(Cp^2\) on \(B_\rho\), uniformly in \(p\). To see this, we define the functions $F(y) = (\frac{1 + y^{-p}}{2})^{-n/p}$ and $y_v(x) = \rho_{K+x}(v)/\rho_K(v)$. Note that:
		\begin{equation}
			\label{eq:taylor_expansion_rho_n}
			\rho_{K_{x, p}}(v)^n = \rho_{K}(v)^n F(y_v(x)).
		\end{equation} 
	The third derivative of \(F\) is given by:
		\begin{align*}
			F'''(y) &= \frac{n(1+y^{-p})^{-n/p}}{y^3(1+y^p)^3}
			\Big(2+n^2+(4+3p-p^2)y^p+ \\ 
			&\qquad\qquad\qquad\qquad\qquad(2+3p+p^2)y^{2p}-3n(1+(1+p)y^p)\Big)
		\end{align*}
	which implies the uniform bound for \(g_{K,p}\).

We also use the following monotonicity. If \(0<q\le p\), then the monotonicity of
power means gives
		\[
			\rho_{K_{x,p}}(v)\le \rho_{K_{x,q}}(v),
			\qquad v\in S^{n-1}.
		\]
Consequently,
		\begin{equation}\label{eq:monotonicity-delta-p}
			g_{K,p}(x)\le g_{K,q}(x),
			\qquad
			\Delta_p(x)\ge \Delta_q(x).
		\end{equation}

We begin with the case \(\alpha>-1\). Since the radial \((-p)\)-mean satisfies
			\[
				\min\{a,b\}
				\le
				\left(\frac{a^{-p}+b^{-p}}2\right)^{-1/p}
				\le
				2^{1/p}\min\{a,b\},
			\]
we have, for every \(x\in\operatorname{int}(-K)\),
			\[
				g_K(x)\le g_{K,p}(x)\le 2^{n/p}g_K(x).
			\]
Thus
			\begin{equation}\label{eq:comparison-gkp-gk}
				0\le g_{K,p}(x)-g_K(x)
				\le
				(2^{n/p}-1)\operatorname{vol}(K)
				=
				O(p^{-1}).
			\end{equation}
			
The classical covariogram expansion gives
			\begin{equation}\label{eq:classical-covariogram-expansion}
				\Delta_\infty(z)
				=
				h_{\Pi K}(z)+O(|z|^2),
				\qquad z\to0,
			\end{equation}
uniformly.

Let
		\[
			E_p=p^{-\alpha}L(p,p^\alpha).
		\]
If \(y\in E_p\), then \(x=p^\alpha y\in L(p,p^\alpha)\), and therefore
		\[
			\Delta_p(p^\alpha y)\leq p^\alpha.
		\]
By \eqref{eq:comparison-gkp-gk},
		\[
			\Delta_\infty(p^\alpha y)
			\le
			\Delta_p(p^\alpha y)+O(p^{-1})
			\le
			p^\alpha+O(p^{-1}).
		\]
Since \(\alpha>-1\), we have \(p^{-1}=o(p^\alpha)\). Hence
		\[
			\Delta_\infty(p^\alpha y)\le C p^\alpha.
		\]
Using the fact that \(h_{\Pi K}\) is a norm and the expansion \eqref{eq:classical-covariogram-expansion}, it follows
that the sets \(E_p\) are eventually contained in a fixed ball.

For \(y\) in a fixed bounded set, \eqref{eq:comparison-gkp-gk}
 and \eqref{eq:classical-covariogram-expansion} give
		\[
			\begin{aligned}
				\frac{\Delta_p(p^\alpha y)}{p^\alpha}
				&=
				\frac{\Delta_\infty(p^\alpha y)}{p^\alpha}
				+
				O(p^{-(\alpha+1)})  \\
				&=
				h_{\Pi K}(y)+O(p^\alpha)+O(p^{-(\alpha+1)}).
			\end{aligned}
		\]
Since \(\alpha<0\) and \(\alpha>-1\), both error terms tend to zero as \(p\to\infty\). Hence
		\[
			\frac{\Delta_p(p^\alpha y)}{p^\alpha}
			\xrightarrow[p\to\infty]{}
			h_{\Pi K}(y)
		\]
locally uniformly in \(y\). Therefore
		\[
			p^{-\alpha}L(p,p^\alpha)
			\xrightarrow[p\to\infty]{d_H}
			\{y:h_{\Pi K}(y)\le1\}
			=
			\Pi^\circ K.
		\]
This proves part  \eqref{thm:level_sets_coupled_convergence_case1}.

We now prove the critical case \(\alpha=-1\). Define
		\[
			F_p(y)
			=
			p\left(
			\operatorname{vol}(K)-g_{K,p}\left(\frac yp\right)
			\right)
			=
			p\Delta_p\left(\frac yp\right).
		\]
Then
		\[
			pL(p,p^{-1})
			=
			\{y\in\mathbb R^n:F_p(y)\leq 1\}.
		\]

Fix \(y\in\mathbb R^n\). For \(v\in S^{n-1}\), set
	\[
		a(v)=\rho_K(v),
		\qquad
		b_p(v)=\rho_{K+y/p}(v),
		\qquad
		x_0(v)=\rho_K(v)v.
	\]
By Lemma~\ref{lem:first_derivative_translated}, applied with \(u=y\), we have
		\[
			\rho_{K+ty}(v)=\rho_K(v)+t\frac{y\cdot n_K(x_0(v))}{v\cdot n_K(x_0(v))}+O(t^2),
		\]
where the error term is uniform in \(v\in S^{n-1}\), and uniform for \(y\) in
compact subsets of \(\mathbb R^n\). Taking \(t=1/p\), we obtain
		\[
			b_p(v)=a(v)+\frac1p\frac{y\cdot n_K(x_0(v))}{v\cdot n_K(x_0(v))}+O(p^{-2}).
		\]
Since \(x_0(v)=a(v)v\), it follows that
		\[
			\frac{b_p(v)}{a(v)}=1+\frac{q_y(v)}p+O(p^{-2}),
		\]
where
		\[
			q_y(v)=\frac{y\cdot n_K(x_0(v))}{x_0(v)\cdot n_K(x_0(v))}.
		\]
Note that the denominator of $q_y(v)$ is bounded away from zero, since it is exactly the support function of $K$ at $n_K(x_0(v))$, which by our assumption that $K$ contains the origin is non-zero.
Therefore
		\[
			\left(\frac{b_p(v)}{a(v)}\right)^{-p}\xrightarrow[p\to\infty]{}e^{-q_y(v)}
		\]
uniformly in \(v\in S^{n-1}\), and uniformly for \(y\) in compact subsets of \(\mathbb R^n\).
To see this, let $R_p(v, y)$ be the remainder term in the expansion of \(\left(\frac{b_p(v)}{a(v)}\right)^{-p}\), which is the $\mathcal{O}(p^{-2})$ term. Taking the logarithm, we have:
\[
	\log\left(\left(\frac{b_p(v)}{a(v)}\right)^{-p}\right) = -p \log\left(1 + \frac{q_y(v)}{p} + R_p(v, y)\right). 
\]
Now, for $p$ large enough, we have that $\left|\frac{q_y(v)}{p} + R_p(v, y)\right| < 1 $, so we can use the Taylor expansion of the logarithm:
\[
	-p\log\left(1 + \frac{q_y(v)}{p} + R_p(v, y)\right) = q_y(v) + \frac{R_p(v, y)}{p} + \mathcal{O}\left(\frac{1}{p}\left(\frac{q_y(v)}{p} + R_p(v, y)\right)^2\right).
\]
Using the fact that \(R_p(v, y) = \mathcal{O}(p^{-2})\) and that \(q_y(v)\) is uniformly bounded for $v$ in \(S^{n-1}\) and \(y\) in compact sets, we conclude that:

\[
	\log\left(\left(\frac{b_p(v)}{a(v)}\right)^{-p}\right) = -q_y(v) + O(p^{-1}), 
\]
and therefore:

\[
	\left(\frac{b_p(v)}{a(v)}\right)^{-p} = e^{-q_y(v)}e^{O(p^{-1})} = e^{-q_y(v)}(1 + O(p^{-1})) = e^{-q_y(v)} + O(p^{-1}), 
\]
with an error term that is independent of $v$ and $y$.

By the definition of the radial \((-p)\)-combination,
		\[
			\rho_{K_{y/p,p}}(v)^{-p}=\frac{\rho_K(v)^{-p}+\rho_{K+y/p}(v)^{-p}}2.
		\]
Hence
		\[
			\rho_{K_{y/p,p}}(v)^n=a(v)^n\left(\frac{1+\left(b_p(v)/a(v)\right)^{-p}}2\right)^{-n/p}.
		\]
Set
		\[
			A_p(v,y)=\frac{1+\left(b_p(v)/a(v)\right)^{-p}}2.
		\]
Then
		\[
			A_p(v,y)\xrightarrow[p\to\infty]{}A(v,y):=\frac{1+e^{-q_y(v)}}2
		\]
uniformly in \(v\in S^{n-1}\), and uniformly for \(y\) in compact subsets of \(\mathbb R^n\).

Using the polar formula for the volume of a star body,
		\[
			g_{K,p}\left(\frac yp\right)=\frac1n\int_{S^{n-1}}a(v)^nA_p(v,y)^{-n/p}\,dv.
		\]
Therefore
		\[
			F_p(y)=\frac pn\int_{S^{n-1}}a(v)^n\left(1-A_p(v,y)^{-n/p}\right)\,dv.
		\]
Since \(A_p(v,y)\) is bounded away from \(0\) and infinity for \(y\) in compact
sets, we have
		\[
			\frac pn\left(1-A_p(v,y)^{-n/p}\right)\xrightarrow[p\to\infty]{}\log A(v,y)
		\]
uniformly in \(v\in S^{n-1}\), and uniformly for \(y\) in compact subsets of \(\mathbb R^n\). Consequently,
		\[
			F_p(y)\xrightarrow[p\to\infty]{}\int_{S^{n-1}}\rho_K(v)^n\log\left(\frac{1+e^{-q_y(v)}}2\right)\,dv
		\]
locally uniformly in \(y\).

We now rewrite the limit. Using
		\[
			\log\left(\frac{1+e^{-q}}2\right)=-\frac q2+\log\cosh\left(\frac q2\right),
		\]
we get
		\[
			\begin{aligned}\int_{S^{n-1}}\rho_K(v)^n\log\left(\frac{1+e^{-q_y(v)}}2\right)\,dv
				&=-\frac12\int_{S^{n-1}}\rho_K(v)^nq_y(v)\,dv \\
				&\quad+\int_{S^{n-1}}\rho_K(v)^n\log\cosh\left(\frac{q_y(v)}2\right)\,dv.
			\end{aligned}
		\]
The first integral vanishes. Indeed, by Lemma~\ref{lem:jacobian_change_of_coordinates}, with \(x=\rho_K(v)v\), we have
		\[
			\rho_K(v)^n\,dv=(x\cdot n_K(x))\,dS(x).
		\]
Moreover,
		\[
			q_y(v)=\frac{y\cdot n_K(x)}{x\cdot n_K(x)}.
		\]
Therefore
		\[
			\int_{S^{n-1}}\rho_K(v)^nq_y(v)\,dv=\int_{\partial K}y\cdot n_K(x)\,dS(x)=0
		\]
by the divergence theorem. Hence
		\[
			F_p(y)\xrightarrow[p\to\infty]{}\Phi_K(y)
		\]
locally uniformly in \(y\), where
		\[
			\Phi_K(y)=\int_{\partial K}(x\cdot n_K(x))\log\cosh\left(\frac{y\cdot n_K(x)}{2x\cdot n_K(x)}\right)\,dS(x).
		\]

It remains to show that the sets \(pL(p,p^{-1})=\{y:F_p(y)<1\}\) are eventually
contained in a fixed ball. From equation \eqref{eq:taylor_delta_p}, after decreasing \(\rho\) if necessary,
there are constants \(c>0\), \(\eta>0\), and \(p_0\ge1\) such that, whenever
\(p\ge p_0\), \(x\in B_\rho\), and \(p|x|\le\eta\), one has
		\begin{equation}\label{eq:quadratic-lower-bound}
			\Delta_p(x)\ge c p|x|^2.
		\end{equation}
		
		Indeed, the error term \(O(p^2|x|^3)\) is then bounded by a sufficiently small
multiple of \(p|x|^2\).

Using \eqref{eq:monotonicity-delta-p} and \eqref{eq:quadratic-lower-bound}, we obtain the following localization estimate: there are
constants \(c_1>0\) and \(\rho_1>0\) such that, for all \(p\) large enough and
all \(x\in B_{\rho_1}\),
		\begin{equation}\label{eq:localization-lower-bound}
			\Delta_p(x)\ge c_1\min\{p|x|^2,|x|\}.
		\end{equation}To see this, if \(p|x|\le\eta\), then \eqref{eq:quadratic-lower-bound} gives the first bound. If
\(p|x|>\eta\), choose \(q=\eta/|x|\), whenever this is at least \(p_0\), and use
\(\Delta_p(x)\ge\Delta_q(x)\) together with \eqref{eq:quadratic-lower-bound} applied to \(q\). The remaining
range is covered by fixing \(q=p_0\) and using the fact that \(\Delta_{p_0}\) is
positive away from the origin.

Now put \(x=y/p\). From \eqref{eq:localization-lower-bound},
		\[
			F_p(y)=p\Delta_p\left(\frac yp\right)\ge c_1\min\{|y|^2,|y|\}
		\]
whenever \(y/p\in B_{\rho_1}\). If \(y/p\notin B_{\rho_1}\), then by \eqref{eq:monotonicity-delta-p}, with
\(q=p_0\), and by compactness away from the origin, we have
		\[
			F_p(y)\to\infty
		\]
uniformly in that region. Therefore there exists \(R>0\) such that
		\[
			pL(p,p^{-1})\subset B_R
		\]
for all sufficiently large \(p\).

Since \(F_p\to\Phi_K\) locally uniformly and the relevant level sets are
uniformly bounded, the sublevel sets converge in the Hausdorff metric. Thus
		\[
			pL(p,p^{-1})\xrightarrow[p\to\infty]{d_H}\{y\in\mathbb R^n:\Phi_K(y)\le1\}=\mathcal T_K.
		\]
This proves case (\ref{thm:level_sets_coupled_convergence_case3}).

Finally, we prove the case \(\alpha<-1\). Let
		\[
			E_p=p^{\frac{1-\alpha}{2}}L(p,p^\alpha).
		\]
Equivalently,
		\[
			E_p=\left\{y:\Delta_p\left(p^{\frac{\alpha-1}{2}}y\right)<p^\alpha\right\}.
		\]

We first show that the sets \(E_p\) are eventually bounded. Put
		\[
			x=p^{\frac{\alpha-1}{2}}y.
		\]
Using \eqref{eq:localization-lower-bound}, for \(x\) sufficiently small we have
		\[
			\Delta_p(x)\ge c_1\min\{p|x|^2,|x|\}.
		\]
In terms of \(y\),
		\[
			p|x|^2=p^\alpha |y|^2,\qquad |x|=p^{\frac{\alpha-1}{2}}|y|.
		\]
If \(y\in E_p\), then \(\Delta_p(x)<p^\alpha\). The previous lower bound implies
that \(|y|\) must remain bounded. The region where \(x\) is not sufficiently small
is excluded as before by monotonicity and compactness away from the origin. Hence
the sets \(E_p\) are eventually contained in a fixed ball.

For \(y\) in a fixed bounded set, we use the Taylor estimate \eqref{eq:taylor_delta_p}:
		\[
			\Delta_p\left(p^{\frac{\alpha-1}{2}}y\right)=\frac{n+p}{8}p^{\alpha-1}y^t\mathcal{M}_Ky+O\left(p^2p^{\frac{3(\alpha-1)}2}|y|^3\right).
		\]
Dividing by \(p^\alpha\), we obtain
		\[
			\frac{\Delta_p\left(p^{\frac{\alpha-1}{2}}y\right)}{p^\alpha}=\frac18y^t\mathcal{M}_Ky+\frac{n}{8p}y^t\mathcal{M}_Ky+O\left(p^{\frac{\alpha+1}{2}}|y|^3\right).
		\]
Since \(\alpha<-1\), we have
		\[
			p^{\frac{\alpha+1}{2}}\xrightarrow[p\to\infty]{}0.
		\]
Therefore
		\[
			\frac{\Delta_p\left(p^{\frac{\alpha-1}{2}}y\right)}{p^\alpha}\xrightarrow[p\to\infty]{}\frac18y^t\mathcal{M}_Ky
		\]
locally uniformly in \(y\). Consequently,
		\[
			p^{\frac{1-\alpha}{2}}L(p,p^\alpha)\xrightarrow[p\to\infty]{d_H}\left\{y\in\mathbb R^n:\frac18y^t\mathcal{M}_Ky\le1\right\}.
		\]
Since
		\[
			\Pi_2^\circ K=\{y\in\mathbb R^n:y^t\mathcal{M}_Ky\le1\},
		\]
the last set is
		\[
			\sqrt8\,\Pi_2^\circ K.
		\]
This finally proves part  \eqref{thm:level_sets_coupled_convergence_case2}.
\end{proof}
		
\begin{remark}[Relation with Orlicz projection bodies]
The transition body \(\mathcal T_K\) can be interpreted in the framework of
Orlicz projection bodies. Let
\[
\varphi(t)=\log\cosh(t/2).
\]
Then
\[
\varphi(0)=0,\qquad
\varphi'(t)=\frac12\tanh(t/2),\qquad
\varphi''(t)=\frac14\operatorname{sech}^2(t/2)>0.
\]
Hence \(\varphi\) belongs to the strictly convex Orlicz class used in the
definition of Orlicz projection bodies.

With this notation, the function \(\Phi_K\) may be written as
\[
\Phi_K(y)
=
\int_{\partial K}
(x\cdot n_K(x))
\varphi\left(
\frac{y\cdot n_K(x)}
{x\cdot n_K(x)}
\right)
\,dS(x).
\]
This is precisely the Orlicz modular that appears in the definition of the
Orlicz projection body, with the normalization of the level set adapted to the
present problem. Equivalently, if we define the gauge
\[
\|y\|_{\varphi,K}
=
\inf\left\{
\lambda>0:
\int_{\partial K}
(x\cdot n_K(x))
\varphi\left(
\frac{y\cdot n_K(x)}
{\lambda\,x\cdot n_K(x)}
\right)
\,dS(x)
\le 1
\right\},
\]
then
\[
\mathcal T_K
=
\{y\in\mathbb R^n:\|y\|_{\varphi,K}\le1\}.
\]
Thus \(\mathcal T_K\) is the polar unit ball associated with the Orlicz
projection-body construction for the Orlicz function
\(\varphi(t)=\log\cosh(t/2)\), up to the choice of normalization in the
functional (which is different from the one in \cite{LutwakYangZhang2010OrliczProjection}). For more on Orlicz bodies we refer to \cite{GardnerHugWeil2014OrliczBM} and to \cite{GardnerHugWeilYe2015DualOrliczBM,ZhuZhouXu2014DualOrliczBM} for its dual theory where radial functions play an important role, in the same way they do in this manuscript. 
\end{remark}		
		
\begin{remark}
	Replacing $K$ by $\beta^{-1/n} K$ where $\beta>0$ in Theorem \ref{thm:level_sets_coupled_convergence}, it is easy to see that
	\[p L(p, \beta p^{-1}) \to \beta^{1/n} \{y \in \mathbb{R}^n : \Phi_{\beta^{-1/n} K}(y) \le \beta\} = \{y \in \mathbb{R}^n : \Phi_K(y) \le \beta\}.\]
	So up to appropriate re-scaling, a transition between the bodies $\Pi^\circ K$ and $\Pi^\circ_2 K$ ocurrs when $\beta \to 0$ and $\beta \to +\infty$.
\end{remark}

\section*{Declarations}

\subsection*{Funding}
The authors were partially supported by the Ministry of Science and Innovation,
Spain, project PID2022-136320NB-I00. The second named author was supported by
Grant RYC2021-031572-I, funded by the Ministry of Science and Innovation /
State Research Agency / 10.13039/501100011033 and by the E.U. Next Generation
EU/Recovery, Transformation and Resilience Plan.

\subsection*{Competing interests}
The authors have no competing interests to declare.

\subsection*{Data availability}
Data sharing is not applicable to this article as no datasets were generated or
analyzed during the current study.

	\bibliographystyle{abbrv}
	
	\bibliography{referencias}

\bigskip\bigskip

\noindent\textsc{Departamento de Análisis Matemático, Universidad de Sevilla}\\
\textsc{C/ Tarfia s/n, Sevilla 41012, Spain}

\medskip

\noindent\textit{Email addresses:}\\
M.~Resende: \texttt{mresende@us.es}\\
J.~Haddad: \texttt{jhaddad@us.es}\\
C.~H.~Jiménez: \texttt{carloshugo@us.es}\\

\end{document}